\documentclass[12pt]{amsart}

\usepackage[all]{xy}
\usepackage{xspace}
\usepackage{graphicx}
\usepackage{longtable}
\usepackage{mathtools}
\usepackage{latexsym}
\usepackage{fullpage,amsfonts,amsmath,amscd,amssymb}
\usepackage{mathabx}
\usepackage{soul}

\input{xy}
\xyoption{all}
\xyoption{curve}

\newcommand{\bA}{{\mathbb A}}

\newcommand{\bC}{{\mathbb C}}

\newcommand{\bF}{{\mathbb F}}

\newcommand{\bK}{{\mathbb K}}

\newcommand{\bN}{{\mathbb N}}

\newcommand{\bZ}{{\mathbb Z}}

\newcommand{\mA}{{\mathcal A}}
\newcommand{\mD}{{\mathcal D}}
\newcommand{\cE}{{\mathcal E}}
\newcommand{\cO}{{\mathcal O}}

\newcommand{\mU}{{\mathcal U}}

\newcommand{\fg}{{\mathfrak g}}

\newcommand{\fh}{{\mathfrak h}}

\newcommand{\fn}{{\mathfrak n}}

\newcommand{\gl}{{\mathfrak g}{\mathfrak l}}
\newcommand{\GL}{{{\mbox{\rm GL}}}}
\newcommand{\Ad}{{{\mbox{\rm Ad}}}}

\newcommand{\Aut}{{{\mbox{\rm Aut}}}}
\newcommand{\Lie}{{{\mbox{\rm Lie}}}}

\newcommand{\PGL}{{{\mbox{\rm PGL}}}} 

\renewcommand{\Pr}{{{\mbox{\rm Prim}}}}

\newcommand{\Exp}{{{\mbox{\rm Exp}}}_{G}}
\newcommand{\Exx}{{{\mbox{\rm Exp}}}}
\newcommand{\Exv}{{{\mbox{\rm Exp}}}_{{\mbox{\tiny GL}}(V)}}

\newcommand{\End}{{{\mbox{\rm End}}}}
\newcommand{\Di}{{{\mbox{\rm Dist}}}}

\newcommand{\Mod}{{-{\mbox{\rm Mod}}}}

\renewcommand{\1}{_{(1)}}
\renewcommand{\2}{_{(2)}}
\newcommand{\0}{_{(n)}}
\newcommand{\8}{^{[p]}}

\newcommand{\va}{\ensuremath{\mathbf{a}}\xspace}

\newcommand{\vd}{\ensuremath{\mathbf{d}}\xspace}
\newcommand{\ve}{\ensuremath{\mathbf{e}}\xspace}

\newcommand{\vh}{\ensuremath{\mathbf{h}}\xspace}

\newcommand{\vx}{\ensuremath{\mathbf{x}}\xspace}
\newcommand{\vy}{\ensuremath{\mathbf{y}}\xspace}

\newtheorem{theorem}{Theorem}
\newtheorem{prop}[theorem]{Proposition}
\newtheorem{lemma}[theorem]{Lemma}

\newtheorem{cor}[theorem]{Corollary}

\newtheorem*{HKC}{Higher Kunz' Conjecture}
\newtheorem*{HFC}{Higher Frobenius Conjecture}

\begin{document}
\title[Integration of Modules]{Integration of Modules -- II: Exponentials}
\author{Dmitriy Rumynin}
\email{D.Rumynin@warwick.ac.uk}
\address{Department of Mathematics, University of Warwick, Coventry, CV4 7AL, UK
  \newline
\hspace*{0.31cm}  Associated member of Laboratory of Algebraic Geometry, National
Research University Higher School of Economics, Russia}
\thanks{The first author was partially supported within the framework of the HSE University Basic Research Program and the Russian Academic Excellence Project `5--100',  as well as by the Max Planck Institute for Mathematics, Bonn.}
\author{Matthew Westaway}
\email{M.P.Westaway@bham.ac.uk}
\address{School of Mathematics, University of Birmingham, Birmingham, B15 2TT, UK
}
\date{July 23, 2018}
  
\begin{abstract}
We continue our exploration of
various approaches to integration of
representations from a Lie algebra $\Lie (G)$
to an algebraic group $G$ 
in positive characteristic.
In the present paper we concentrate on
an approach exploiting exponentials.
This approach works well for over-restricted representations,
introduced in this paper, and takes no note of
  $G$-stability.
\end{abstract}

\maketitle

If $G$ is a connected simply-connected Lie group
with Lie 
algebra $\fg$,
the categories of finite-dimensional $G$-modules and
$\fg$-modules are equivalent.
In one direction the equivalence
is a differentiation functor 
$\mD : G\Mod \rightarrow \fg\Mod$.
Its quasi-inverse is an integration (or exponentiation) functor
$\cE : \fg\Mod \rightarrow G\Mod$.
Since every $x\in G$ can be written as a product of exponentials
$x= \Exp(\va_1)\Exp(\va_2) \ldots \Exp (\va_n)$, $\va_i\in \fg$,
we can 
exponentiate a representation
$\cE (V,\theta) = (V,\Theta)$
by an explicit formula
$$
\Theta \Big( \Exp(\va_1)\Exp(\va_2) \cdots \Exp (\va_n)\Big) =
\Exv(\theta(\va_1)) \cdots \Exv (\theta(\va_n)).
$$

This method works for a semisimple simply-connected algebraic group $G$
over $\bC$ and its category
$G\Mod$ of rational representations.
The key observation is that $\cE (V,\theta)$ is rational.
In this case $G$ is generated by unipotent root subgroups $U_\alpha$,
thus, we can choose $\va_i\in\fg_{\alpha_i}$ in the exponentiation formula.
Then $\theta (\va_i)$ is nilpotent, so $\Exv (\theta (\va_i))$ is polynomial.

Curiously, we can use the same formula for exponentiation of representations
for more general algebraic groups. However, we
can no longer rely on the Lie group structure on $G$ for proving
that the exponentiation formula produces a well-defined group
homomorphism $\Theta : G \rightarrow \GL (V)$. It is a minor inconvenience
in zero characteristic that turns into a major technical issue in positive
characteristic.

The idea of using exponentials
in positive characteristic goes back to Chevalley and his construction
of finite groups of Lie type.
Kac and Weisfeiler use exponentials in positive characteristic
to study contragradient Lie algebras \cite{VK}. 
If $G$ is an algebraic group over a field of positive characteristic,
its Lie algebra $\fg$ is a restricted subalgebra of the commutator
Lie algebra $U_0(\fg)^{(-)}$ of the restricted enveloping algebra
$U_0(\fg)$.
N.B., $U_0(\fg)^{(-)}$ is the Lie algebra
of the algebraic group $\GL_1 (U_0(\fg))$
but $\fg$ is not an algebraic subalgebra, i.e.,
not the Lie algebra of an algebraic subgroup. 
Let $\widehat{G} \leq \GL_1 (U_0(\fg))$
be 
a minimal (possibly non-unique) algebraic subgroup\footnote{We have been informed that there exists an unpublished preprint by Haboush where a group similar to $\widehat{G}$ has been studied but we could not find it.}
  of $\GL_1 (U_0(\fg))$
whose Lie algebra contains $\fg$. 

An interesting fact about the group $\widehat{G}$ is that it acts
on restricted $\fg$-modules. If we can relate the groups $\widehat{G}$
and $G$, we may  be able to integrate representations of $\fg$.
Notice that not all $\fg$-modules are integrable: the baby Verma modules
(except Steinberg modules) have
non-$G$-stable support varieties, hence, cannot be integrated.
Yet $\widehat{G}$ acts on the baby Verma modules. This suggests that the
relation between $G$ and $\widehat{G}$ is delicate.  
We uncover this relation for some class of modules which we call {\em over-restricted}.

Now we reveal the detailed content of the present paper,
emphasising the main results. 
We study exponentials on a restricted representation
of a restricted Lie algebra in Section~\ref{s1.1}.
These are particularly well-behaved
when the representation is not only restricted
but also {\em over-restricted}, a concept
introduced in this section.
Our first major result of the paper is Theorem~\ref{extension_g} 
in this section. This yields
Corollary~\ref{extension_g1}, 
a notable general result
which says that for an algebraic group $G$,
with some mild restrictions, an over-restricted representation
of its Lie algebra can be integrated to the group $G$.

In Section~\ref{s1.N} we extend the concept
of an over-restricted representation to
Kac-Moody Lie algebras.
The main result of this section is Theorem~\ref{extension_gKM}:
an over-restricted representation of a Kac-Moody 
Lie algebra can be integrated to the Kac-Moody group.
The set-up of this section is similar, yet slightly different
from the over-restricted representations of Kac-Moody
algebras discussed by the first author
in another paper \cite{Ru}. 

In Section~\ref{s1.2} we study over-restricted representations
of the higher Frobenius kernels of a semisimple algebraic group $G$.
We switch to semisimple groups as there are some subtleties to overcome
compared to the first Frobenius kernels.
We stop short of proving 
an analogue of Theorem~\ref{extension_g}
for the higher Frobenius kernels.
We formulate it as a conjecture instead.


In Section~\ref{s1.3} we elaborate how our Higher Frobenius Conjecture
applies to the  
Humphreys-Verma Conjecture, a well-known hypothesis that
projective $U_0(\fg)$-modules are $G$-modules.

We discuss examples of over-restricted representations in Section~\ref{s1.4}.
We give several non-trivial examples of over-restricted representations
of classical simple Lie algebras and propose the notion of an over-restricted
enveloping algebra.

We draw conclusions for this and the first paper in the series \cite{RuW}
in Section~\ref{s3.5}. The final section~\ref{Appendix} is a technical appendix
with essential results on generic smoothness of morphisms of algebraic varieties.

We would like to thank Ami Braun for valuable discussions
and pointing our attention to Kunz' Conjecture.
We are grateful to Jim Humphreys for helpful comments
and detection of some  mistakes in early versions.
We are indebted to Stephen Donkin
for encouragement, interest in our work and sharing 
Xanthopoulos' thesis.

\section{Over-restricted Representations}
\label{s1.1}
Let $(\fg, \,\8)$ 
be a restricted Lie algebra over a field $\bK$ of characteristic $p$,
$U_0 (\fg)$ its restricted enveloping algebra, 
$(V,\theta)$ a restricted representation.
Let $N_p(\fg)$ be the $p$-nilpotent cone of $\fg$, i.e,
the set of all $\vx\in\fg$ such that $\vx\8 =0$.
Notice that for $\vx\in N_p(\fg)$
we have $\theta (\vx)^p = \theta (\vx\8) = 0$. 
This allows us to define
exponentials for each $\vx\in N_p(\fg)$:
$$
e^{\theta(\vx)} = \sum_{k=0}^{p-1}
\frac{1}{k!} \theta(\vx)^k \in {\mathfrak g}{\mathfrak l} (V)
\; .
$$
The element $e^{\theta(\vx)}$ is 
invertible
because
$(e^{\theta(\vx)})^{-1}= e^{\theta(-\vx)}$.
We define
{\em the pseudo-Chevalley group} $G_V$
as the subgroup of $\GL (V)$
generated by all exponentials $e^{\theta(\vx)}$ for all $\vx\in N_p(\fg)$.
\begin{prop} \label{prop_GV}
  The following statements hold for any restricted
  finite-dimensional representation $(V,\theta)$ of $\fg$:
  \begin{enumerate}
  \item $G_V$ is a (Zariski) closed subgroup of $\GL (V)$.
  \item One can choose finitely many
    $\vx_1,\vx_2, \ldots, \vx_n \in N_p (\fg)$
    such that the following map $f$ is surjective: 
    $$
    f: \bK^n \rightarrow G_V, \ \
    f(a_1,a_2, \ldots , a_n) =
  e^{\theta (a_1 \vx_1)}\cdots e^{\theta (a_n \vx_n)} \; .
  $$
  \end{enumerate}
\end{prop}
\begin{proof}
It follows from the standard fact \cite[Prop I.2.2]{Bor}
  by choosing $I=N_p(\fg)$, $V_\vx = \bK$, $f_\vx (a) = e^{\theta (a \vx)}$
  in Borel's notation.
\end{proof}

Two particular pseudo-Chevalley groups are worth separate discussion.
Let $(U_0(\fg),\theta)$ be the left regular representation of $\fg$
on its restricted enveloping algebra. The exponential
$e^{\theta (\vx)}$
is uniquely determined by its application to the identity
$$
e^{\theta (\vx)} (1) =
\sum_{k=0}^{p-1} \frac{1}{k!} \vx^k \in U_0(\fg)\; .
$$
This element should be called 
$e^{\vx} \in U_0(\fg)$.
We can identify
$e^{\theta (\vx)}$
with $e^{\vx}$
because 
$G_{U_0(\fg)}$
is a subgroup of $\GL_1 (U_0(\fg))$
that, in turn,
acts on $U_0(\fg)$ by left multiplication:
$$
G_{U_0(\fg)} \leq \GL_1 (U_0(\fg)) \leq \GL (U_0(\fg)_{\bK}).
$$
We define the group $\widehat{G}$
discussed in the introduction
as $\widehat{G}\coloneqq G_{U_0(\fg)}$.
It acts on restricted $\fg$-modules,
hence, its structure is worth further investigation.

The element $e^\vx$ is not group-like in $U_0(\fg)$, yet it is close to it in
the sense that
$$
\Delta (e^\vx) =
e^\vx \otimes e^\vx + \cO (\vx^{\lfloor(p+1)/2\rfloor})
$$
where $\cO (\vx^{m})$ denotes a sum of terms containing $\vx^k$ with $k\geq m$. 
To make this precise, we 
say that a $U_0 (\fg)$-module $V$ is {\em over-restricted} if
$\theta (\vx)^{\lfloor(p+1)/2\rfloor} = 0$ for all $\vx\in N_p(\fg)$.
See Section~\ref{s1.4} for some examples. 
Notice that if $p=2$, then $\lfloor(p+1)/2\rfloor = 1$
and this requirement is severe: $\theta (\vx)=0$.
\begin{prop}
\label{abs_chev}
Let 
$(\fg,ad)$ be the adjoint representation.
If 
$(V,\theta)$ is an
over-restricted representation,
then
$$
\theta(e^{ad(\vx)}(\vy))= e^{\theta(\vx)}\theta (\vy) e^{-\theta(\vx)}
$$ 
for all $\vx\in N_p(\fg)$, $\vy\in\fg$.
\end{prop}
\begin{proof}
First, observe by induction that for each $k=1,2,\ldots, p-1$
$$
\theta(\frac{1}{k!}ad(\vx)^k(\vy)) =
\sum_{j=0}^k \frac{(-1)^j}{(k-j)!j!} \theta(\vx)^{k-j}\theta(\vy)\theta(\vx)^j
\;
.
$$
For $k=1$ this is just the definition of a representation:
$$
\theta(ad(\vx)(\vy)) = \theta([\vx,\vy]) = \theta(\vx)\theta(\vy)- \theta(\vy)\theta(\vx)
.
$$
Going from $k$ to $k+1$, 
\begin{align*}
\theta(\frac{1}{(k+1)!}ad(\vx)^{k+1}(\vy)) = 
&
\frac{1}{k+1} ( \theta (\vx) \theta(\frac{1}{k!}ad(\vx)^k(\vy))- \theta(\frac{1}{k!}ad(\vx)^k(\vy)) \theta (\vx))
\\
= \sum_{j=0}^k \frac{(-1)^j}{k+1} 
\bigg(\frac{1}{(k-j)!j!} \theta(\vx)^{k-j+1}\theta(\vy)\theta(\vx)^j 
& -
\frac{1}{(k-j)!j!} \theta(\vx)^{k-j}\theta(\vy)\theta(\vx)^{j+1}\bigg)
\\
= \frac{1}{(k+1)!}\theta(\vx)^{k+1}\theta(\vy) +
&
\sum_{i=1}^{k} \frac{(-1)^i}{(k+1)(k-i)!(i-1)!}
\bigg(\frac{1}{i} +  
\frac{1}{k+1-i}\bigg) \cdot
\\
\cdot\, 
\theta(\vx)^{k+1-i}\theta(\vy)\theta(\vx)^{i}
+ \frac{(-1)^{k+1}}{(k+1)!}\theta(\vy)\theta(\vx)^{k+1}
&
=
\sum_{i=0}^{k+1} \frac{(-1)^i}{(k+1-i)!i!} \theta(\vx)^{k+1-i}\theta(\vy)\theta(\vx)^i \; .
\end{align*}

Finally,
\begin{align*}
\theta(e^{ad(\vx)}(\vy))= 
\sum_{k=0}^{p-1} \theta(\frac{1}{k!}ad(\vx)^k(\vy))
&
=
\sum_{i+j=0}^{p-1} \frac{(-1)^j}{i!j!} \theta(\vx)^{i}\theta(\vy)\theta(\vx)^j
=
\\
\sum_{i,j=0}^{p-1} \frac{(-1)^j}{i!j!} \theta(\vx)^{i}\theta(\vy)\theta(\vx)^j
&
=
\Big( \sum_{i=0}^{p-1} \frac{1}{i!}
\theta(\vx)^{i}\Big)
\theta(\vy)
\sum_{j=0}^{p-1} \frac{(-1)^j}{j!} \theta(\vx)^j
=
e^{\theta(\vx)}\theta (\vy) e^{-\theta(\vx)},
\end{align*}
where the third equality holds because $(V,\theta)$ is over-restricted:
all missing terms are actually zero.
\end{proof}

The second vital example of a pseudo-Chevalley group
is $G_\fg$, procured from the adjoint representation $(\fg, ad)$.
This group is intricately connected with
the pseudo-Chevalley groups of over-restricted representations:
\begin{prop}\label{surj_hom}
If $(V,\theta)$ is a faithful 
over-restricted representation of $\fg$, 
then the assignment 
$$
\phi: e^{\theta (N_p (\fg))}\rightarrow G_\fg, \ \
\phi (e^{\theta (\vx)}) = e^{ad(\vx)},  \ \
\vx\in N_p(\fg)
$$
extends to a surjective homomorphism of groups
$\phi: G_V\rightarrow G_\fg$
whose kernel is central and consists of $\fg$-automorphisms
of $V$.
\end{prop}
\begin{proof}
Proposition~\ref{prop_GV} yields
the elements $\vx_1, \ldots, \vx_n \in N_p (\fg)$
for $G_V$ and the elements $\vx_{n+1}, \ldots, \vx_m \in N_p (\fg)$
for $G_\fg$.
Combining these elements together, we get surjective algebraic maps
with common domain:
    $$
    f: \bK^m \rightarrow G_V, \
    \widehat{f}: \bK^m \rightarrow G_\fg, \ \
    f\Big( (a_k) \Big) = \prod_k e^{\theta (a_k \vx_k)}, \
    \widehat{f}\Big( (a_k) \Big) = \prod_k e^{ad(a_k \vx_k)}\; .
  $$
Let   $H= (\bK,+)^{\ast m}$ be the free product of $m$ additive groups.
The maps $f$ and $\widehat{f}$ extend to surjective group homomorphisms
    $$
    f^\sharp: H \rightarrow G_V, \
    \widehat{f}^\sharp : H \rightarrow G_\fg
    $$
so that both $G_V$ and $G_\fg$ are quotients of $H$ as abstract groups.
Consider an element of the kernel
$a_1\ast\ldots \ast a_k \in \ker (f^\sharp)$
where $a_i$ belongs to the $t(i)$-th component of the free product.
Clearly, 
$$
I_V =
f^\sharp (a_1\ast\ldots \ast a_k)
=
e^{\theta (a_1\vx_{t(1)})}e^{\theta (a_2\vx_{t(2)})} \ldots e^{\theta (a_k\vx_{t(k)})} 
\; .
$$ 
Proposition~\ref{abs_chev} tells us 
that 
$$
\theta(e^{ad(a_1\vx_{t(1)})}e^{ad(a_2\vx_{t(2)})}\ldots e^{ad(a_k\vx_{t(k)})}(\vy))= \theta (\vy)
\ 
\mbox{ for all }
\
\vy \in \fg.$$
Since $\theta$ is injective it follows that 
$e^{ad(a_1\vx_{t(1)})}\ldots e^{ad(a_k\vx_{t(k)})} = I_\fg$,   
so $a_1\ast\ldots \ast a_k \in \ker (\widehat{f}^\sharp)$. 
It follows that the homomorphism $\phi$ is well-defined.

Consider $A=
e^{\theta (a_1\vx_{t(1)})} \ldots e^{\theta (a_k\vx_{t(k)})} 
\in \ker (\phi)$.
By
Proposition~\ref{abs_chev},
$\theta(\vy) = \theta (\phi (A) (\vy)) = A\theta (\vy) A^{-1}$
for all $\vy\in\fg$.
Hence, $A\in \Aut_\fg (V)$,
so that $A$ commutes with all $\theta(\vy)$.
Consequently, $A$ commutes
with all $e^{\theta (\vx)}$, which are generators of $G_V$. Hence, $A$ is central.
\end{proof}

It is natural to inquire whether the homomorphism $\phi$
is a homomorphism of algebraic groups.
To prove this,
we need a technical result, Theorem~\ref{generic_smooth}
about
generic smoothness of polynomial
maps in positive characteristic,
established in the appendix (Section~\ref{Appendix}).
We include the answer to this natural question
in the main result of this section:

\begin{theorem} \label{extension_g}
  Suppose that 
  the field $\bK$ is algebraically  closed.
  The following statements hold for a faithful 
  over-restricted finite-dimensional representation $(V,\theta)$ of
  a finite-dimensional restricted Lie algebra $\fg$:
  \begin{enumerate}
  \item The map $\phi: G_V\rightarrow G_\fg$ 
    constructed in Proposition~\ref{surj_hom}  
    is a homomorphism of algebraic groups.
  \item The Lie algebra $\Lie (G_V)$ is isomorphic to $\theta (\fg_0)$
    where $\fg_0$ is the Lie subalgebra of $\fg$, 
    generated by all $\vx\in N_p(\fg)$.
    Moreover, $\fg_0$ is a restricted Lie subalgebra of $\fg$. 
  \item The differential $\vd_1 \eta$ of the natural representation
    $\eta: G_V\hookrightarrow \GL(V)$ is equal to
    $\theta |_{\fg_0}$.
  \item The differential $\vd_1 \phi$ is surjective.
    Its kernel is $\fg_0 \cap Z( \fg)$ where $Z(\fg)$ is the centre.
    \item The scheme-theoretic kernel $\ker \phi$
    is a subgroup scheme of $\Aut_\fg (V)$, central in $G_V$.
  \item If $Z(\fg)=0$,
    then $\ker\phi$ is discrete.
      \end{enumerate}
\end{theorem}
\begin{proof}
  (1) On top of the surjective maps
  $f: \bK^m \rightarrow G_V$
  and
  $\widehat{f}: \bK^m \rightarrow G_\fg$,
  utilised in  Proposition~\ref{surj_hom},
by \cite[Prop I.2.2]{Bor}  
we can find 
    $\vx_{m+1},\vx_{m+2} \ldots, \vx_k \in N_p (\fg)$
such that the image $G$ of the map 
    $$
    \widetilde{f}: \bK^k \rightarrow G_V \times G_{\fg}, \ \
    f(a_1,a_2, \ldots , a_k) =
(e^{\theta (a_1 \vx_1)}\cdots e^{\theta (a_k \vx_k)} , e^{ad (a_1 \vx_1)}\cdots e^{ad (a_k\vx_k)}) 
  $$
is a closed algebraic subgroup of $G_V \times G_{\fg}$. 
Extending $f$ and $\widehat{f}$ in the obvious way to maps
$f^\prime$ and $\widehat{f}^\prime$ defined on $\bK^k$,
we see that   $\widetilde{f} = (f^\prime,\widehat{f}^\prime)$.
Hence, $G$ is the graph of the group homomorphism $\phi : G_V \rightarrow G_\fg$.

Moreover, the first projection $\pi_1 : G \rightarrow G_V$
  is bijective. 
Since $f^\prime$ is given by polynomials
  of degree less than $p$ by construction,
  Theorem~\ref{generic_smooth}
  ensures that $f^\prime$ is generically smooth.
  Since $\vd \pi_1 \circ \vd \widetilde{f} = \vd f^\prime$,
  the differential $\vd \pi_1$ is surjective at some point.
  Since $\pi_1$ is a morphism of algebraic groups,
  the differential $\vd \pi_1$ is surjective at all points.
  Hence, $\pi_1$ is an isomorphism of algebraic groups.
  Consequently, $\phi$ is a morphism of algebraic varieties (or groups) 
  since $\phi = \pi_2 \pi_1^{-1}$. 

  (2)
  Let $\fg_1$ be the linear span of all $\vx\in N_p(\fg)$.
  Let $(z_1, \ldots , z_k)$ be the standard coordinates on $\bK^k$.
  For all $i=1, \ldots k$   the calculation
$$
\vd_0 f^\prime (\frac{\partial}{\partial z_i})=
\frac{d}{dt}e^{\theta (t\vx_i)}\vert_{t=0} = \theta (\vx_i)
$$
implies that $\Lie (G_V) \supseteq
{\mathrm{Im}} (\vd_0 f^\prime ) = \theta (\fg_1)$.
It follows that 
$\Lie (G_V) \supseteq \theta (\fg_0)$.
  
By  Theorem~\ref{generic_smooth}, 
the differential $\vd_a f^\prime$ is surjective at some point $a\in \bK^k$.
If $L_a : G_V \rightarrow G_V$ is the left multiplication by $f^\prime (a)^{-1}$,
  then the Lie algebra $\Lie (G_V)$ is spanned by the elements
\begin{align*}  
\vd_{f^\prime(a)} L_a\big(  \vd_a f^\prime (\frac{\partial}{\partial z_i})\big) =
\vd_{f^\prime(a)} L_a\big(\frac{d}{dt}e^{\theta (a_1\vx_1)}\ldots
&
e^{\theta (a_{i-1}\vx_{i-1})}e^{\theta ((a_i+t)\vx_i)}e^{\theta (a_{i+1}\vx_{i+1})}\ldots  \vert_{t=0} \big) =
\\
\vd_{f^\prime(a)} L_a\big(
e^{\theta (a_1\vx_1)}\ldots e^{\theta (a_{i-1}\vx_{i-1})}e^{\theta (a_i\vx_i)}\theta(\vx_i)
&
e^{\theta (a_{i+1}\vx_{i+1})}\ldots \big)
=
e^{-\theta (a_n\vx_n)}\ldots
e^{-\theta (a_{i+1}\vx_{i+1})}\theta(\vx_i)
\cdot
\\
e^{\theta (a_{i+1}\vx_{i+1})}\ldots e^{\theta (a_{n}\vx_{n})}
=
\theta \big( e^{-ad (a_n\vx_n)}\ldots
&
e^{-ad (a_{i+1}\vx_{i+1})} (\vx_i) \big). 
\end{align*}
The last equality holds because of Proposition~\ref{abs_chev}. 
The element
$e^{-ad (a_n\vx_n)}\ldots
e^{-ad (a_{i+1}\vx_{i+1})} (\vx_i)$
belongs to $\fg_0$ since all $\vx_j$ belong there.
Hence, this calculation shows $\Lie (G_V) \subseteq \theta (\fg_0)$.

It remains to argue that $\fg_0$ is a restricted Lie subalgebra of $\fg$.
This is true because $\theta$ is an injective
homomorphism of restricted Lie algebras,
and both $\theta (\fg_0)=\Lie (G_V)$  and $\theta (\fg)$
are restricted subalgebras of $\gl (V)$. 

(3) It follows from the same 
calculation as just above for $\vx\in N_p (\fg)$:
  $$
\vd_1 \eta (\vx) = 
  \frac{d}{dt}e^{\theta (t\vx)}\vert_{t=0} = \theta (\vx). 
  $$
  
  (4)
  The same argument as in (1) shows that $\vd_1 \pi_2$ is surjective.
  Hence, $\vd_1 \phi = \vd_1\pi_2 \circ \vd_1 \pi_1^{-1}$
  is surjective as well.

  The second statement follows from the observation that
  $\vd_1 \phi = ad|_{\fg_0}$. This can be checked on elements
  $\vx \in N_p (\fg)$ since they span $\fg_0$:
    $$
  \vd_1 \phi (\vx)=
  \frac{d}{dt}e^{ad (t\vx)}\vert_{t=0} = ad (\vx)\, . 
  $$

  (5) This follows from
Proposition~\ref{surj_hom}.

(6) It follows from (4) that  
the differential $\vd_1\phi : \Lie(G_V)\rightarrow \Lie (G_\fg)$
is an isomorphism of Lie algebras.
Observe that $G_V$ is connected because it is generated as a group
by a connected set $e^{\theta(N_p (\fg))}$ containing the identity element. 
Hence, the kernel of $\phi$ is discrete.
\end{proof}

Let us state an immediate, rather curious corollary
of the proof of part~{(2)}:
\begin{cor} Let $\fg$ be a finite-dimensional restricted Lie algebra
  over an algebraically closed field that admits a faithful over-restricted
  representation. Let $\fg_1$ be the span of $N_p (\fg)$.
  The following statements in the notation
  of the proof of Theorem~\ref{extension_g}{.(2)} are equivalent:
  \begin{enumerate}
  \item $\fg_1$ is a restricted Lie subalgebra,
  \item for some choice of $\theta$ and $f^\prime$, the differential
    $\vd_0 f^\prime$ is surjective,
  \item for all choices of $\theta$ and $f^\prime$, the differential
    $\vd_0 f^\prime$ is surjective.
  \end{enumerate}  
\end{cor}

  
Our terminology of pseudo-Chevalley groups
is justified by the following example:
consider the adjoint representation $\fg$ of a semisimple algebraic
group $G$.
Then,
barring accidents in small characteristic, (for instance, if $p\geq 5$), 
$G_\fg$ is precisely the adjoint Chevalley group $G_{ad}$.
Notice that the Chevalley group $G_{ad}$
is generated by the exponentials of root vectors $\ve_\alpha$.
In characteristic zero $ad_{\bZ} (\ve_\alpha)^4=0$, while
in positive characteristic $ad (\ve_\alpha)^p=0$
so the exponentials could be different. For instance,
if $G$ is of type $G_2$ in characteristic 3, then
the Chevalley exponential $e_{\bZ}^{\ve_\alpha}$
of the short root vector $\ve_\alpha$ contains
the divided-power term $ad_\bZ (\ve_\alpha^{(3)})$
but our exponential stops at $ad(\ve_\alpha)^2/2$.
Similar difficulty appears for all groups in characteristic 2.
It would be interesting to investigate this question further:
what is the precise relation between $G_\fg$ and $G_{ad}$
for simple algebraic groups in characteristic 2
(and the type $G_2$ group in characteristic 3).

Let us contemplate applications of Theorem~\ref{extension_g}
to integration of representations.
Suppose $\fg = \Lie (G)$
where $G$ is a connected algebraic group
(over an algebraically closed field $\bK$). 
The adjoint group $G_{ad}$ is defined as
the image of 
the adjoint representation $\Ad: G \rightarrow \GL (\fg)$.
Notice that $G_{ad}$ is closed because
the image of a morphism of algebraic groups is closed \cite[I.1.4]{Bor}.
We can compare  $G_{ad}$ and $G_\fg$
as sets because both are algebraic subgroups of $\GL (\fg)$.

\begin{cor} \label{extension_g1}
  Suppose that $G_{ad} = G_\fg$. 
  The following statements hold for a faithful 
  over-restricted finite-dimensional representation $(V,\theta)$ of
  $\fg = \Lie (G)$:
  \begin{enumerate}
  \item The representation $(V,\theta)$ yields a rational
    representation $(V,\Theta)$
    of a central extension (that happens to be $G_V$)
    of $G_{ad}$ such that $\vd_1 \Theta (\vx)= \theta (\vx)$
    for all $\vx \in \fg_0$. 
  \item If $(V,\theta)$ is a brick (i.e., $\End_\fg V = \bK$),
    then $(V,\theta)$ yields a rational
    projective representation of $G_{ad}$
   such that $\vd_1 \Theta (\vx)= \theta (\vx)$
    for all $\vx \in \fg_0$. 
  \end{enumerate}
\end{cor}

We finish the section with an application to semisimple groups.
Notice that it is true in characteristic 2 because
over-restricted representations are direct sums of the trivial representation.
\begin{cor} \label{extension_g11}
Suppose that $G$ is a connected simply-connected
  semisimple algebraic group such that $Z(\fg)=0$.
  Assume further that if $p=3$, then
  $G$ has no components of type $G_2$.
Then a faithful   over-restricted finite-dimensional representation $(V,\theta)$ of $\fg$
integrates to a rational representation of $G$.
\end{cor}

\section{Kac-Moody Groups}\label{s1.N}
Let $\mA=(A_{i,j})_{n\times n}$ be a generalised Cartan matrix,
$\fg_\bC = \fg_\bC (\mA)$ its corresponding complex
Kac-Moody algebra.
The divided powers integral form
$\mU_\bZ$ 
of the universal enveloping algebra
$U(\fg_\bC)$
forges the Kac-Moody algebra over any commutative ring $\bA$: 
$$
\fg_\bZ \coloneqq \fg_\bC \cap \mU_\bZ \; , \ \ 
\fg_\bA \coloneqq \fg_\bZ \otimes_\bZ \bA \; . 
$$
It inherits a triangular decomposition
$\fg_\bA =
(\fn_{-} \otimes \bA)
\oplus
(\fh \otimes \bA)
\oplus
(\fn_{+} \otimes \bA)$
from
$\fg_\bZ =
\fn_{-}
\oplus \fh \oplus
\fn_{+}$.
If $\bK$ is a field of characteristic $p$,
the Lie algebra $\fg_\bK$
is restricted with the $p$-operation
$$
(\vh\otimes 1)^{[p]} = \vh\otimes 1, \ \ 
(\vx\otimes 1)^{[p]} = \vx^p \otimes 1
\ \mbox{ where } \
\vh\in \fh, \; \vx\in \fn_\pm
$$
where $\vx^p$ is calculated inside the associative $\bZ$-algebra
$\mU_\bZ$
\cite[Th. 4.39]{Marquis}.
In particular, $(\ve_\alpha\otimes 1)^{[p]} =0$
for any real root vector $\ve_\alpha$.

{\em The Kac-Moody group} is a functor $G_\mA$ from commutative rings to groups.
Its value on a field $\bF$ can be described
using the set of real roots $\Phi^{re}$: 
$$
G_\mA (\bF)
\; = \;
\Asterisk_{\alpha\in\Phi^{re}} \, U_\alpha
/ \langle \;\mbox{Tits' relations}\;\rangle
, \ \
U_\alpha = \{ X_\alpha (t)\, | \, t\in \bF \} \cong  \bF^+.
$$
There are different ways to write Tits' relations: the reader should consult
the classical papers \cite{CaCh,Tits}
for succinct presentations.

While the precise relations are peripheral to our deliberations,
the following fact is vital: 
\begin{center}
{\em the group $G_\mA (\bF)$ acts on the Lie algebra $\fg_\bF$ via the adjoint
action} \cite{Marquis,Remy}.
\end{center}
The adjoint action of each root subgroup $U_\alpha$ is exponential
over $\bZ$, reduced to the field $\bF$:
$$
\Ad (X_\alpha (t)) (\va \otimes 1)
=
`` \; \Exx (ad (\ve_\alpha \otimes t)) (\va \otimes 1) \; "
\coloneqq 
\sum_{n=0}^\infty \big( \, \frac{1}{n!} \, ad (\ve_\alpha)^{n} (\va) \otimes t^n \big).
$$
Observe that the latter sum is well-defined:
if $\va\in\fg_\bZ$ then $\frac{1}{n!} \, ad (\ve_\alpha)^{n} (\va)\in\fg_\bZ$.
The sum is actually finite:
by writing $\va = \sum_\beta \va_\beta$ as a sum of elements
from root subspaces we can see that there exists $N$ such that
$n\alpha +\beta$ is not a root for all $n>N$ and all $\beta$ so that, consequently,
$ad (\ve_\alpha)^{n} (\va) =0$ as soon as $n>N$. 
We denote the image of $\Ad$ by $G^{ad}_\mA (\bF)$
and call it {\em the adjoint Kac-Moody group}.

Let $\bK$ be a field of positive characteristic $p$.
Each real root $\alpha$ yields an additive family
of linear operators
(in the sense that $Y_\alpha (t+s)=Y_\alpha (t)Y_\alpha (s)$) 
on a restricted representation
$(V,\theta)$ of the Lie algebra $\fg_\bK$:
$$
Y_\alpha (t) \coloneqq e^{\theta (\ve_\alpha \otimes t)}
= \sum_{k=0}^{p-1} \frac{1}{k!}\, \theta (\ve_\alpha \otimes t)^k.
$$
By $G^{KM}_V$ we denote the group generated by $Y_\alpha (t)$
for all real roots $\alpha$ and $t\in\bK$.
Notice that  $G^{KM}_V$  is a subgroup of $G_V$, defined in Section~\ref{s1.1}.
If $p>\max_{i\neq j} (-A_{ij})$, then $\fg_\bK$ is generated by root vectors $\ve_\alpha$
\cite{Rou2} and, consequently, we expect that $G^{KM}_V=G_V$
for all over-restricted faithful representations.
It is an interesting problem to compare
$G^{KM}_V$  and $G_V$ for an arbitrary representation. 
If $(V,\theta)$ is over-restricted,
then Proposition~\ref{abs_chev}
applies:
\begin{equation}
  \label{eq1}
\theta \big( \Ad (X_{\alpha}(t))(\vy) \big)
=
Y_\alpha (t)
\theta (\vy)
Y_\alpha (-t) 
\end{equation}
for all $\vy\in\fg_\bK$.
Here is the main result of this section,
which is an adaptation of Proposition~\ref{surj_hom}
(cf. \cite[Theorem 1.2]{Ru} for a graded version of this result):
\begin{theorem} \label{extension_gKM}
If $(V,\theta)$ is a faithful
over-restricted representation of $\fg_\bK$,
then the assignment
$\phi (Y_\alpha (t) ) = \Ad (X_\alpha (t))$
extends to a surjective homomorphism of groups
$
\phi: G^{KM}_V\rightarrow G^{ad}_\mA (\bK) 
$, 
whose kernel is central and consists of $\fg_\bK$-automorphisms
of $V$.
\end{theorem}
\begin{proof}
  Let 
  $H$ be the free product of all additive groups
  $U_\alpha$ for all real roots $\alpha$.
  Both $G^{KM}_V$ and $G^{ad}_\mA(\bK)$ are naturally quotients of $H$.
  From this point the rest of the proof repeats
  the proof of Proposition~\ref{surj_hom} word for word.
%
\end{proof}

As soon as there are few endomorphisms, the map $\phi$ in Theorem~\ref{extension_g}
can be ``reversed'' to define a projective representation of the Kac-Moody group.

\begin{cor}
If in the conditions of Theorem~\ref{extension_g}
  the representation 
  $(V,\theta)$ is a brick (see Cor.~\ref{extension_g1}), 
then the assignment
$$
\Theta: G^{ad}_\mA (\bK) \rightarrow \GL (V), \
\Theta (\Ad (X_\alpha (t) ) = Y_\alpha (t),
$$
extends to a group homomorphism
$G^{ad}_\mA (\bK) \rightarrow \PGL (V)$
and, thus, defines a projective representation of $G^{ad}_\mA (\bK)$.
\end{cor}

\section{Higher Frobenius Kernels}
\label{s1.2}
%
%
In this section we take $G$ to be a semisimple simply-connected split
algebraic group over a field $\bK$ of characteristic $p>0$.
Let $\Phi$ be the root system of $G$,
$\Pi=\{\alpha_1,\ldots,\alpha_r\}\subseteq\Phi$ a basis of simple roots.
The standard Chevalley basis of the Lie algebra
$\fg = \mbox{Lie} (G)$ is $\ve_\alpha, \ \alpha \in\Phi$,
$\vh_i = [\ve_{\alpha_i}, \ve_{-\alpha_i}]$.
In particular, $\fg$ is generated by $\ve_\alpha, \; \alpha\in \Phi$.
It is useful to keep in mind that $ad(\ve_\alpha)^p=0$ for all $\alpha \in\Phi$.

Let $G\0$ be the $n$-th Frobenius kernel of $G$,
$\Di (G\0)$ the distribution algebra on it.
$\Di(G\0)$ has a divided powers basis
$$
\prod_{\alpha\in \Phi^{+}}\ve_\alpha^{(m_\alpha)}
\prod_{\beta\in\Pi}\binom{\vh_\beta}{n_\beta}
\prod_{\alpha\in\Phi^{+}}\ve_{-\alpha}^{(m_{-\alpha})}
\ \ \ 0\leq m_\alpha,n_\beta,m_{-\alpha}<p^n \; . 
$$
If $k<p$ then
$$
\ve^{(k)}=\frac{1}{k!}\ve^k
\in \Di(G\1) \ni
\binom{\vh}{k}=\frac{1}{k!} \vh (\vh-1) \ldots (\vh-k+1)
$$
so that $\Di (G\1)$ is a subalgebra of $\Di (G\0)$, naturally isomorphic to
$U_0(\fg)$.

Let us now consider
a representation $(V,\theta)$ of $G\0$.
It is naturally 
a representation of $\Di(G\0)$ which we also denote by $(V,\theta)$.
We define exponentials in an analogous way to the previous section:
$$
Y_{\alpha}(t) = Y^V_{\alpha}(t)
\coloneqq e^{\theta(t\ve_{\alpha})}=\sum_{k=0}^{p^n-1}\theta(t^k\ve_{\alpha}^{(k)})
\in \End(V), \ \
Z_{\alpha}(t)=e^{t\ve_\alpha}=\sum_{k=0}^{p^n-1}t^k\ve_\alpha^{(k)}\in \Di(G\0)
$$
where $t\in\bK$ and $\alpha\in \Phi$.
Both $Y_\alpha(t)$ and $Z_\alpha(t)$ are invertible.
In fact, 
these are one-parameter subgroups:
$Y_\alpha(t)Y_\alpha(s)=Y_\alpha(t+s)$ and
$Z_\alpha(t)Z_\alpha(s)=Z_\alpha(t+s)$.
Let us generate subgroups by them:	
$$
G_{(n),V} \coloneqq \langle Y_\alpha (t) \,\mid\,
\alpha\in\Phi, t\in \bK \rangle \leq \GL(V), \ 
\widetilde{G}\coloneqq \langle Z_\alpha (t) \,\mid\, \alpha\in\Phi, t\in \bK \rangle \leq \GL_1(\Di(G\0)). 
$$
Conjugation by $G$ equips $\Di(G\0)$ with a $G$-module structure,
which we can then restrict to $G\0$-module and $\Di(G\0)$-module structures.
We denote the corresponding representation of $\Di(G\0)$ by $ad$ because
it is a version of the adjoint representation; for instance,
the ``usual'' adjoint representation on $\fg$ is a subrepresentation under
$\fg \hookrightarrow U_0 (\fg) \hookrightarrow \Di (G\0)$
(cf. \cite[I.7.18, I.7.11(4)]{Jan}). We also use $ad$ to denote the representation of $\Di(G)$ on $\Di(G\0)$; this restricts to the above $ad$ on $\Di(G\0)$.
We say that $(V,\theta)$ is {\em $n$-over-restricted} if
$\theta(\ve_\alpha^{(k)})=0$ for all $k\geq \lfloor (p^n+1)/2\rfloor$ and all $\alpha\in\Phi$.
Notice that if $p^n=2$ then this condition forces
$(V,\theta)$ to be a direct sum of copies of the trivial module.
\begin{prop}
  \label{abs_n_chev} 
  (cf. Proposition~\ref{abs_chev}) 
          If $(V,\theta)$ is an $n$-over-restricted
          representation of $\Di(G\0)$, then
          $$
          \theta \Big(ad(Z_\alpha(t))(\vd)\Big) = Y_\alpha(t)\theta(\vd)Y_\alpha(-t)
          $$
	for all $t\in \bK$, $\alpha\in\Phi$ and $\vd\in \Di(G\0)$.
	\end{prop}
	\begin{proof}
          We write $ad$ using Sweedler's $\Sigma$-notation \cite[I.7.18]{Jan}:  
          $$
          ad(\vx)(\vd)=\sum_{(\vx)}\vx\1 \vd S(\vx\2) \ \mbox{ for all } \vx,\vd \in \Di (G\0).
          $$
Since $\Delta(\ve_\alpha^{(k)})=\sum_{i+j=k}\ve_\alpha^{(i)}\otimes\ve_\alpha^{(j)}$ and
          $S(\ve_\alpha^{(k)})=(-1)^k\ve_\alpha^{(k)}$, we get
$$\theta(ad(t^k\ve_\alpha^{(k)})(\vd))=
\theta(\sum_{i+j=k}(-1)^jt^k\ve_\alpha^{(i)}\vd\ve_\alpha^{(j)})=
\sum_{i+j=k}\theta(t^i\ve_\alpha^{(i)})\theta(\vd)\theta((-t)^j\ve_\alpha^{(j)}).
$$
Hence,
$$
\theta \Big(ad(Z_\alpha(t))(\vd)\Big) 
=
\sum_{k=0}^{p^n-1}\sum_{i+j=k}\theta(t^i\ve_\alpha^{(i)})\theta(\vd)\theta((-t)^j\ve_\alpha^{(j)}). 
$$
On the other hand, we have 
$$
Y_\alpha(t)\theta(\vd)Y_\alpha(-t)
=
\sum_{i,j=0}^{p^n-1}\theta(t^i\ve_\alpha^{(i)})\theta(\vd)\theta((-t)^j\ve_\alpha^{(j)})
.
$$
The result follows from the fact that $V$ is $n$-over-restricted.
	\end{proof}

        It is useful to remind the reader that $\fg$ can be recovered inside
        $\Di (G\0)$ as the set of primitive elements:
        $$
        \fg = \Pr (\Di (G\0)) \coloneqq
        \{ \vd \in \Di (G\0)\;\mid\; \Delta(\vd) = \vd\otimes 1 + 1 \otimes \vd \}.$$
        This explains why $\fg$ is a
        submodule of $\Di (G\0)$ under the adjoint action: we leave it to the reader
        to check that $ad(\vx)(\vd)\in \Pr (\Di (G\0))$ for all
        $\vx\in\Di (G\0)$ and $\vd\in \Pr (\Di (G\0))$. 
        
	\begin{prop}\label{surh_map}
          Let $(V,\theta)$ be an $n$-over-restricted representation of $\Di(G\0)$,
          faithful on $\fg$. 
Then the assignment
		$$
\phi (Y^V_\alpha(t)) =Y^{\fg}_\alpha(t) \ (\; = e^{ad(t\ve_\alpha)})$$
extends to		
a surjective homomorphism of groups
$\phi: G_{(n),V}\rightarrow G_{(n),\fg}$, whose kernel consists of $\fg$-automorphisms
		of $V$.                
	\end{prop}
\begin{proof}
	The fact that $\phi$ is a well-defined homomorphism
  is proved in a similar way as
  in Proposition~\ref{surj_hom}.
Let $H=\ast_{\alpha}U_\alpha$ be the free product of (additive) root subgroups.
  Both $G_{(n),V}$ and $G_{(n),\fg}$ are naturally quotients of $H$.
If $W_{\beta_1}(t_1)\ast\ldots \ast W_{\beta_m}(t_m) \in \ker (H\rightarrow G_{(n),V})$ then
$$
Y^V_{\beta_1}(t_1)\ldots Y^V_{\beta_m}(t_m) 
= I_V
\; . 
$$ 
Proposition~\ref{abs_n_chev} tells us 
that for all $\vd \in \fg$
$$
\theta(ad(Z_{\beta_1}(t_1))ad(Z_{\beta_2}(t_2))\ldots ad(Z_{\beta_m}(t_m))(\vd))=
\theta(Y^{\fg}_{\beta_1}(t_1)\ldots Y^{\fg}_{\beta_m}(t_m)(\vd))= 
\theta (\vd). $$
Since $\theta$ is faithful on $\fg$, 
$Y^{\fg}_{\beta_1}(t_1)Y^{\fg}_{\beta_2}(t_2)\ldots Y^{\fg}_{\beta_m}(t_m)
=
I_\fg$,   
hence
$W_{\beta_1}(t_1)\ast\ldots \ast W_{\beta_m}(t_m) 
\in \ker (H\rightarrow G_{(n),\fg})$. 
Thus, the homomorphism $\phi$ is well-defined.

Suppose
$A=Y^V_{\beta_1}(t_1)\ldots Y^V_{\beta_m}(t_m) 
\in \ker (\phi)$.
By above, $\theta(\vd)=\theta(\phi(A)(\vd))=A\theta(\vd)A^{-1}$ for all
$\vd\in \fg$. Hence, $A\in \Aut_{\fg}(V)$.
			\end{proof}

If the adjoint representation is $n$-over-restricted, 
we can identify the adjoint group $G_{ad}$ with $G_{(n),\fg}$.
Proposition~\ref{surh_map}
yields an exact sequence of abstract groups
$$
1\rightarrow Z_{(n),V} \rightarrow
G_{(n),V}\xrightarrow{\phi} G_{ad}
\rightarrow 1
$$
where $Z_{(n),V}$ is the kernel of $\phi$.
To tie up loose ends
we need to address the algebraic group properties
of this sequence:
\begin{HFC} \label{extension_gn}
  Suppose that $G$ is a semisimple connected algebraic group
  over an algebraically  closed field $\bK$.
  The  following statements should hold
for an $n$-over-restricted finite-dimensional representation $(V,\theta)$ of $G\0$,
  faithful on $\fg$:
  \begin{enumerate}
  \item The map $\phi: G_{(n),V}\rightarrow G_{(n),\fg}$
    constructed in Proposition~\ref{surh_map}
    is a homomorphism of algebraic groups.
  \item 
If $(\fg,ad)$ is $n$-over-restricted then
$\phi: G_{(n),V}\rightarrow G_{(n),\fg}$
is a central extension of algebraic groups.
  \item If $(\fg,ad)$ is $n$-over-restricted then
    $(V,\theta)$
    extends to a rational representation of the simply-connected group $G_{sc}$.
     \end{enumerate}
\end{HFC}


\section{Applications of Higher Frobenius Conjecture}
\label{s1.3}
We consider $G$ as in the previous section, and assume $\bK$ to be algebraically closed. 
Let $(P,\theta)$ be a projective indecomposable $U_0 (\fg)$-module.
The well-known Humphreys-Verma Conjecture \cite{Ball, Don2, HuVe, Sob}
(currently proved for $p\geq 2h-2$, where $h$ is the Coxeter number \cite{Jan0}, cf. \cite[II.11.11]{Jan})
states $(P,\theta)$ extends to a $G$-module.
A similar statement for the higher Frobenius kernels follows from Humphreys-Verma
Conjecture
\cite[Remark II.11.18]{Jan}. 
Let us examine what 
the Higher Frobenius Conjecture can contribute towards this long-standing conjecture.

Let $T$ be a maximal torus of $G$.
$TG\0$-modules are the same as $X(T)$-graded $G\0$-modules.
We can control
the condition of being $n$-over-restricted for them
by monitoring their weights $X(V)=\{ \lambda\in X(T) \,\mid\, V_\lambda \neq 0 \}$.
We define {\em the height of } $V$ by the following formula:
$$
\xi (V) \coloneqq \inf \{ n \in \bN \,\mid\,
\forall \alpha \in \Phi \ \ \ X(V)\cap (X(V)+n\alpha) = \emptyset \}.
$$
Clearly $\theta(\ve_\alpha^{(\xi (V))})=0$ is guaranteed for a $TG\0$-module $(V,\theta)$.
Hence, the next proposition immediately follows from the Higher Frobenius Conjecture:
\begin{prop}
  \label{extension_grade}
 Suppose that 
Higher Frobenius Conjecture holds for 
a connected simply-connected
semisimple algebraic group $G$ such that $Z(\fg)=0$.
  Assume further that if $p^n=3$, then
  $G$ has no components of type $G_2$.
  Let $(V,\theta)$ be a $TG\0$-module, faithful as a $\fg$-module,
  such that $p^n \geq 2 \xi (V)-1$ if $p$ is odd, or $p^n\geq 2\xi(V)$ if $p=2$.
  Then $(V,\theta)$ can be extended to a $G$-module.
\end{prop}

It follows that if a $TG\1$-module
can be extended to a $TG\0$-module for sufficiently large $n$,
then it can be extended to a $G$-module. Due to particular significance of projective
$U_0 (\fg)$-modules we state this observation for them as a proposition. 
Recall that $\rho=\frac{1}{2}\sum_{\alpha\in\Phi^{+}}\alpha$ is the half-sum
of positive roots.
Let $a=\max_{1\leq i\leq r}(a_i)$
where 
  $2\rho=\sum_{\alpha_i\in\Pi}a_i\alpha_i$ for $a_i\in \bZ$.

\begin{prop}
  \label{extension_p}
 Suppose that the Higher Frobenius Conjecture holds for 
a connected simply-connected
semisimple algebraic group $G$ such that $Z(\fg)=0$.
  Let $P$ be a projective indecomposable $U_0 (\fg)$-module.
  Suppose $P$ extends to a rational $G\0$-module where
  $$
  n \geq \log_p(4a(p-1)+1).
  $$
  if $p$ is odd, or 
  $$n\geq \log_2(a+1)+2$$
  if $p=2$.
  Then $P$ extends to a $G$-module.
\end{prop}
\begin{proof}
  It is known that $P$ is a $TG\1$-module \cite[II.11.3]{Jan}.
  Clearly, $\xi (P) \leq \xi (U_0 (\fg))$.
  From the PBW-basis, it follows that the ``top'' grade of the grading on $U_0(\fg)$
  is attained by the element $\prod_{\alpha\in\Phi^{+}}\ve_\alpha^{p-1}$.
  This has grade $2(p-1)\rho$.
  Similarly, the ``bottom'' grade
  is $-2(p-1)\rho$.
  Thus, $\xi (U_0 (\fg))\leq 2(p-1)a+1$
and the condition in Proposition~\ref{extension_grade}, when $p$ is odd, becomes
$
p^n \geq 2\xi (U_0 (\fg))-1$; for this to be true, it is enough that $p^n\geq4a(p-1)+1.
$
When $p=2$, the condition becomes $2^{n-1}\geq \xi(U_0 (\fg))$, for which it is enough that $2^{n-1}\geq 2a+1$ or equivalently $2^{n-2}\geq a+1$.
\end{proof}

For the reader's benefit we add two tables.
The first contains the values of $2h-2$ and $a$.
The second lists the smallest prime $p_0$ for all groups up to rank 8
so that 
extension of $P$ to a rational $G\0$-module guarantees an extension
to a rational $G$-module
as soon as $p\geq p_0$
(the column is the type of $G$,
the row is $G\0$).
It also lists the smallest $n$ such that extension to $G\0$ ensures extension
to $G$ for $p=2,3,5$.
Some of the entries are marked with the \linebreak
dagger $\,^{\mbox{\tiny \textdagger}}$.
This signifies the presence of a nontrivial centre $Z(\fg)\neq 0$.

\begin{table}
	\label{ta1}
	\begin{center}
		\caption{Coxeter numbers and coefficients $a$}
		\vskip 3mm
		\bgroup
		\renewcommand{\arraystretch}{2}
		\begin{tabular}{|c||c|c|c|c|c|c|c|c|c|c|}
			\hline 
			& $A_{2l+1}$ & $A_{2l}$ & $B_n$ & $C_n$ & $D_{n}$ & $E_6$ & $E_7$ & $E_8$ & $F_4$ & $G_2$\\
			\hline
			$2h-2$ & $4l+2$ & $4l$ & $4n-2$ & $4n-2$ & $4n-6$ & $22$ & $34$ & $58$ & $22$ & $10$ \\
			$a$ & $(l+1)^2$ & $l(l+1)$ & $n^2$ & $(n-1)(n+2)$ & $(n+1)(n-2)$ & $42$ & $96$ & $270$ & $42$ & $10$ \\
			\hline 
		\end{tabular}
		\egroup
	\end{center}
\end{table}


  \begin{table}
	\label{ta2}
	\begin{center}
		\caption{$G_{(n)}$-extension requirements in characteristic $p$}
		\vskip 3mm
		\bgroup
		\def\arraystretch{2}
		\begin{tabular}{|c||c|c|c|c||c|}
			\hline 
			& $G_{(2)}$ & $G_{(3)}$ & $G_{(4)}$ & $G_{(5)}$ & $2$ \\
			\hline \hline
			$A_1$ & 3 & $\,^{\mbox{\tiny \textdagger}} 2$ & $\,^{\mbox{\tiny \textdagger}} 2$ & $\,^{\mbox{\tiny \textdagger}} 2$ & $\,^{\mbox{\tiny \textdagger}} G_{(3)}$ \\
			\hline
			$A_2$ & 7 & $\,^{\mbox{\tiny \textdagger}} 3$ & 2 & 2 & $G_{(4)}$ \\
			\hline
			$B_2$ & 17 & 5 & 3 & $\,^{\mbox{\tiny \textdagger}} 2$ & $\,^{\mbox{\tiny \textdagger}} G_{(5)}$ \\
			\hline
			$G_2$ & 41 & 7 & 3 & 3 & $G_{(6)}$ \\
			\hline
			$A_3$ & 17 & 5 & 3 & $\,^{\mbox{\tiny \textdagger}} 2$ & $\,^{\mbox{\tiny \textdagger}} G_{(5)}$ \\
			\hline
			$B_3$ & 37 & 7 & 3 & 3 & $\,^{\mbox{\tiny \textdagger}} G_{(6)}$ \\
			\hline
			$C_3$ & 41 & 7 & 3 & 3 & $\,^{\mbox{\tiny \textdagger}} G_{(6)}$ \\
			\hline
			$A_4$ & 23 & $\,^{\mbox{\tiny \textdagger}} 5$ & 3 & 2 & $G_{(5)}$ \\
			\hline
			$B_4$ & 67 & 11 & 5 & 3 & $\,^{\mbox{\tiny \textdagger}} G_{(7)}$ \\
			\hline
			$C_4$ & 71 & 11 & 5 & 3 & $\,^{\mbox{\tiny \textdagger}} G_{(7)}$ \\
			\hline
			$D_4$ & 41 & 7 & 3 & 3 & $G_{(6)}$ \\
			\hline
			$A_5$ & 37 & 7 & $\,^{\mbox{\tiny \textdagger}} 3$ & $\,^{\mbox{\tiny \textdagger}} 3$ & $G_{(6)}$ \\
			\hline
			$B_5$ & 101 & 11 & 5 & 3 & $\,^{\mbox{\tiny \textdagger}} G_{(7)}$ \\
			\hline
			$C_5$ & 113 & 11 & 5 & 3 & $\,^{\mbox{\tiny \textdagger}} G_{(7)}$ \\
			\hline
			$D_5$ & 71 & 11 & 5 & 3 & $G_{(7)}$ \\
			\hline
		\end{tabular}
		\quad
		\begin{tabular}{|c||c|c|c|c||c|c|c|}
			\hline 
			& $G_{(2)}$ & $G_{(3)}$ & $G_{(4)}$ & $G_{(5)}$ & $2$ & $3$ & $5$ \\
			\hline \hline
			$F_4$ & 167 & 13 & 7 & 5 & $G_{(8)}$ & $G_{(6)}$ & $G_{(5)}$ \\
			\hline
			$A_6$ & 47 & $\,^{\mbox{\tiny \textdagger}} 7$ & 5 & 3 & $G_{(6)}$ & $G_{(5)}$ & $G_{(4)}$\\
			\hline
			$B_6$ & 149 & 13 & 5 & 5 & $\,^{\mbox{\tiny \textdagger}} G_{(8)}$ & $G_{(6)}$ & $G_{(4)}$ \\
			\hline
			$C_6$ & 161 & 13 & 7 & 5 & $\,^{\mbox{\tiny \textdagger}} G_{(8)}$ & $G_{(6)}$ & $G_{(5)}$ \\
			\hline
			$D_6$ & 113 & 11 & 5 & 3 & $G_{(7)}$ & $G_{(5)}$ & $G_{(4)}$ \\
			\hline
			$E_6$ & 167 & 13 & 7 & 5 & $G_{(8)}$ & $\,^{\mbox{\tiny \textdagger}} G_{(6)}$ & $G_{(5)}$ \\
			\hline
			$A_7$ & 67 & 11 & 5 & 3 & $\,^{\mbox{\tiny \textdagger}} G_{(7)}$ & $G_{(5)}$ & $G_{(4)}$ \\
			\hline
			$B_7$ & 193 & 17 & 7 & 5 & $\,^{\mbox{\tiny \textdagger}} G_{(8)}$ & $G_{(6)}$ & $G_{(5)}$ \\
			\hline
			$C_7$ & 221 & 17 & 7 & 5 & $\,^{\mbox{\tiny \textdagger}} G_{(8)}$ & $G_{(6)}$ & $G_{(5)}$ \\
			\hline
			$D_7$ & 161 & 13 & 7 & 5 & $G_{(8)}$ & $G_{(6)}$ & $G_{(5)}$ \\
			\hline
			$E_7$ & 383 & 23 & 7 & 5 & $\,^{\mbox{\tiny \textdagger}} G_{(9)}$ & $G_{(7)}$ & $G_{(5)}$ \\
			\hline
			$A_8$ & 79 & 11 & 5 & 3 & $G_{(7)}$ & $\,^{\mbox{\tiny \textdagger}} G_{(5)}$ & $G_{(4)}$ \\
			\hline
			$B_8$ & 257 & 17 & 7 & 5 & $\,^{\mbox{\tiny \textdagger}} G_{(9)}$ & $G_{(6)}$ & $G_{(5)}$ \\
			\hline
			$C_8$ & 281 & 17 & 7 & 5 & $\,^{\mbox{\tiny \textdagger}} G_{(9)}$ & $G_{(6)}$ & $G_{(5)}$ \\
			\hline
			$D_8$ & 221 & 17 & 7 & 5 & $G_{(8)}$ & $G_{(6)}$ & $G_{(5)}$ \\
			\hline
			$E_8$ & 1087 & 37 & 11 & 7 & $G_{(11)}$ & $G_{(7)}$ & $G_{(6)}$ \\
			\hline	
		\end{tabular}
		\egroup
	\end{center}
\end{table}

\section{Examples}
\label{s1.4}
The heights can be computed for Weyl modules.
Let $V(\lambda)$ be the Weyl module with the highest weight
$\lambda = \sum_i k_i \varpi_i$ written in the basis of fundamental weights.
It follows from the description of $V(\lambda)$ by generators and relations
\cite[Theorem 21.4]{Hu2}
that
$$
\xi (V(\lambda)) \leq 1 + 2 \max_{i} \frac{(\lambda,\alpha_i)}{(\alpha_i,\alpha_i)}
=
1 + \max_{i} k_i \; . 
$$
This means that the Weyl modules with $k_i\leq(p-1)/2$ for all $i=1,\ldots,r$ are 
over-restricted.
For instance, if $\fg$ is of type $A_2$ 
then (for $p>3$)
the Weyl module 
$V(\frac{p-1}{2} \omega_1 +  \frac{p-1}{2} \omega_2)$
is the only over-restricted Weyl module outside
the first closed $p$-alcove (under the $\bullet$-action):
indeed, 
$k_1+k_2 = p-1 > p-2$.
Thus, most (but not all) over-restricted modules are semisimple in this case.

On the other hand, if $\fg$ is of type $G_2$ and $\alpha_1$ is short, 
then the over-restricted Weyl module
$V(\frac{p-1}{2} \omega_1 +  \frac{p-1}{2} \omega_2)$
lies
inside the ninth 
$p$-alcove (if $p>3$): 
$$
k_1+2k_2 = \frac{3}{2}(p-1) < 2p-3, \ 
k_1+3k_2 = 2(p-1) > 2p-4, \ 
k_1 = \frac{p-1}{2} < p-1.
$$
Ninth in this context means that there are eight dominant $p$-alcoves
below it. 
Thus, in type $G_2$ there are many over-restricted non-semisimple modules.

It is an interesting problem
to achieve a detailed description of over-restricted modules. 
We can formulate some precise questions if we consider
{\em the over-restricted enveloping algebra} 
$$
U_{\rm \tiny over} (\fg) 
\coloneqq
U_0(\fg) / \langle \ve_\alpha^{\lfloor (p+1)/2 \rfloor}, \; \alpha \in \Phi \rangle.
$$
What is the  centre of $U_{\rm \tiny over} (\fg)$? 
Can we describe the blocks of
$U_{\rm \tiny over} (\fg)$  by quivers with relations? 
Which of its blocks are tame and which are finite?

\section{Conclusion}
\label{s3.5}
What have we achieved in this and the preceding paper \cite{RuW}? Suppose $G$ is a semisimple algebraic group with Lie algebra $\fg$. Which concrete $\fg$-modules can we now extend to $G$-modules?
One evident case is when $(V,\theta)$ is an indecomposable $G$-stable $\fg$-module such that $G$ acts trivially on $\Aut_\fg (V,\theta)$.
By combination of \cite[Corollary 23]{RuW}, \cite[Lemma 25]{RuW}
and the cohomology vanishing of the trivial module \cite[II.4.11]{Jan},
$H^2_{Rat}(G,G\1;A)=0=H^1_{Rat}(G,G\1;A)$ for all 
$A$, constituents of $\Aut_\fg (V,\theta)$. 
Thus, the $\fg$-module structure of such $(V,\theta)$ extends uniquely to a $G$-module structure.

It is possible to ensure triviality of the action if one can control the weights.
The weights of simple constituents of  $\Aut_\fg (V,\theta)$ must be divisible by $p$ because
$G\1$ acts trivially. On the other hand, the weights of $V\otimes V^\ast$ are differences of weights of $V$.
Thus, we have 
a version of
Proposition~\ref{extension_grade}:  
\begin{prop}
	\label{extension_grade_weak}
	Let $(V,\theta)$ be a $G$-stable $TG\1$-module
	such that $p \geq 2 \xi (V)-1$.
	Then $(V,\theta)$ can be uniquely extended to a $G$-module.
\end{prop}
It would be interesting to extend this result to the higher Frobenius kernels. 

%

\section{Appendix: Generic Smoothness in Positive Characteristic}
\label{Appendix}
A morphism $\Psi:X\rightarrow Y$
of irreducible algebraic varieties
over an algebraically closed field 
is called {\em smooth}
if   $d_\vx \Psi : T_\vx X \rightarrow T_{\Psi(\vx)} Y$ is surjective for all $\vx\in X$.  
The morphism $\Psi:X\rightarrow Y$
is called {\em generically smooth}
if there exists a dense open subset $U\subseteq X$ such that
$d_\vx \Psi$ is surjective for all $\vx\in U$.

A generically smooth morphism is necessarily dominant.
In the opposite direction,
it is a standard fact that dominant morphisms are generically smooth
in zero characteristic
\cite[II.6.2 Lemma 2]{Shaf},
but it is manifestly untrue in positive characteristic.
For instance, 
the Frobenius morphism,
e.g., $\Psi(\vx)=\vx^p$ from the affine line to itself,
has zero differential at every point.

The issue is best understood on the rational level.
Let $\bK (X)$ be the field of rational functions on the variety $X$.
\begin{lemma} \cite[Prop. AG.17.3]{Bor}
\label{ap_sep}
  Let
  $\Psi:X\rightarrow Y$
  be a dominant morphism 
  of irreducible algebraic varieties
  over an algebraically closed field $\bK$.
  Then $\Psi$ is generically smooth
  if and only if 
  the pullback field extension
  $\bK(Y) \lhook\joinrel\xrightarrow{\Psi^\sharp}\bK(X)$
  is separable.
\end{lemma}

Our aim is to contemplate a polynomial map
$$
F = (F_j(x_1, \ldots x_n))_{j=1}^m: \bK^n \rightarrow \bK^m.
$$
\begin{lemma} \label{decomposition}
  Let $Y$ be the Zariski closure of the image of
  the polynomial map $F$. 
  Then there exist a dense Zariski-open set $U\subset \bK^n$,
  a sequence of varieties $U_0=U,U_1,\ldots,U_k$,
  a sequence of algebraic morphisms
  $H_t : U_t \rightarrow U_{t+1}$ for $t=0,\ldots, k-1$
and an algebraic morphism 
  $\widetilde{F} : U_k \rightarrow Y$
such that
\begin{enumerate}
\item on $U$ the map $F$ factors as $F|_U = \widetilde{F}\circ H_{k-1} \circ \ldots H_0$,
  \item for each $t$ the map $H_t$ is finite of degree $p$ and purely inseparable, 
  \item the morphism $\widetilde{F} : U_k \rightarrow Y$ is smooth. 
\end{enumerate}
\end{lemma}
\begin{proof}
  Let $x_1,\ldots,x_n$ be the coordinate functions on $\bK^n$,
  $z_1,\ldots,z_m$
  the pull-backs to $\bK^n$ of the coordinate functions on $\bK^m$. 
%
Consider a maximal (in $\bK(x_1,\ldots,x_n)$) separable extension
$\widetilde{\bK} \supset F^{*}\bK(Y)=\bK(z_1,\ldots,z_m)$. Hence, the gap extension 
$\bK(x_1,\ldots,x_n)\supset \widetilde{\bK}$
is purely inseparable. It can be decomposed as a tower of degree $p$
purely inseparable extensions
$$
\bK_0=\bK(x_1,\ldots,x_n)\supset
\bK_1 \supset \cdots \supset
\bK_{k-1} \supset \bK_k = \widetilde{\bK}. 
$$
For each intermediate extension we can pick an element $y_t\in \bK_0$
such that $y_t^p\in \bK_t$ and $\bK_{t-1}=\bK_{t} (y_t)$. 

Now the field $\widetilde{\bK}$ is finitely-generated,
so suppose $\widetilde{\bK}=\bK(w_1,\ldots,w_l)$
where the elements $w_j$ are not necessarily algebraically independent.
Let $A_0$ be the subalgebra of $\bK_0$ generated by all $w_j$, $x_j$ and $y_j$.
Its spectrum is an open subset of $\bK^n$. Let us define
$A_t \coloneqq A_0 \cap \bK_t$. Let us examine the towers of algebras
and their  quotient fields
$$
A_0\supset
A_1 \supset \cdots \supset
A_k 
\ \mbox{ and } \ 
Q(A_0) = \bK_0 \supset
Q(A_1) \supset \cdots \supset
Q(A_{k-1}) \supset Q(A_k)= \bK_k. 
$$
While the equalities
$Q(A_0) = \bK_0$
and 
$Q(A_k)= \bK_k$
follow from our construction, in general, only
$Q(A_t) \subseteq \bK_t$ can be immediately discerned. 
Notice, however, that $y_t\in A_{t-1}$ but
$y_t \not\in \bK_t \supseteq Q(A_t)$.
Thus, all extensions in the tower of the quotient fields are proper.
Inevitably, by degree considerations, $Q(A_t)= \bK_t$ for all $t$. 

The spectra of the rings $A_t$ and the algebraic maps defined by their inclusions,
which we denote $H_t$, nearly satisfy the requirements of the lemma.
The only issue is that the map $\mbox{Spec}(A_k) \rightarrow Y$ is only generically smooth.
Let $U_k$ be a dense open subset of $\mbox{Spec}(A_k)$ where this map is smooth.
It remains to define all the varieties recursively: $U_{t}\coloneqq H_t^{-1} (U_{t+1})$. 
  \end{proof}

Now we have a tool to establish the key property:
``small degree'' polynomial maps are generically smooth.
\begin{theorem}
  \label{generic_smooth}
  Suppose that each degree
  ${\mathrm{Deg}}_{x_t} (F_j  (x_1, \ldots x_n))$
  of every component of a polynomial map
$F = (F_j(x_1, \ldots x_n))_{j=1}^m: \bK^n \rightarrow \bK^m$
is less than $p$. 
    Let $Y$ be the Zariski closure of the image of
  the polynomial map $F$. 
  Then   the corestricted morphism $\widehat{F}\coloneqq F|^Y : \bK^n \rightarrow Y$ is generically smooth. 
\end{theorem}
\begin{proof}
  Since the function $\vx\mapsto {\mathrm{Rank}}\, d_\vx \widehat{F}$ is lower semicontinuous,
  it suffices to find a single point $\vx \in \bK^n$ where the differential $d_\vx \widehat{F}$ is surjective.

  Lemma~\ref{decomposition} yields the varieties $U_t$, the maps $H_t$, as well as various rational
  functions $w_t$, $x_t$, $y_t$ and $z_t$. Near any  point $\vx\in U_0$ we can choose
  local parameters $X_t \coloneqq x_t - x_t (\vx)$ so that the formal neighbourhood of $\vx$ in $U_0$
  is the formal spectrum of $B_0 = \bK [[ X_1, \ldots , X_n]]$. If $\widehat{F} (\vx)$ is smooth,
  we can choose local parameters near $\widehat{F} (\vx)$ from the coordinate functions $z_1, \ldots, z_m$ on $\bK^m$.
  Without loss of generality, the local parameters are
  $Z_t \coloneqq z_t - z_t (F(\vx))$ for $t=1,\ldots, s$ where $s = \dim Y \leq m$. In particular,
  $$
  Z_t 
  = F_t(x_1,\ldots, x_n) - z_t (F(\vx))
  = F_t(X_1 + x_1 (\vx) ,\ldots, X_n + x_n (\vx)) - z_t (F(\vx))
  = \widetilde{F_t}(X_1,\ldots, X_n),
  $$
  where $\widetilde{F_t}$ is a polynomial without a free term of degree less than $p$ in each variable
  so that near a generic $\vx$ the map $\widehat{F}$ is described by the embedding
  $$B_0 = \bK [[ X_1, \ldots , X_n]] \supseteq B_{\infty} \coloneqq \bK [[ Z_1, \ldots , Z_s]]$$
  on the level of formal neighbourhoods.
  
  For a generic point $\vx\in U_0$ all of its images $\vx_t = H_{t-1} (H_{t-2} \cdots H_0 (\vx)\cdots )\in U_t$
  are smooth. Let $B_t$ be the ring of functions on the formal neighbourhood of $\vx_t$, i.e.,
  the  formal neighbourhood is the formal spectrum of $B_t$.
  Since $\vx_t$ is smooth, the ring
  $B_t$ is the ring of formal power series: $B_t  \cong \bK [[ X_1, \ldots , X_n]]$.
  Let us examine the tower of formal neighbourhoods
  $$
B_0 = \bK [[ X_1, \ldots , X_n]] \supset
B_1 \supset \cdots \supset
B_k 
\supset
B_{\infty} = \bK [[ Z_1, \ldots , Z_s]]. 
$$
In the notation of Lemma~\ref{decomposition} we can observe that
$\bK_t^p \subseteq \bK_{t+1}$. It follows that for a generic $\vx$
we have the same inclusion on the formal level: $B_t^p \subseteq B_{t+1}$ for all $t<k$.
As a corollary of the Kimura-Niitsima Theorem \cite[Cor. 2]{KiNi}
(cf. \cite[Section 15 and Exercise 15.4]{Kunz}), we can describe each map $H_t$ on the formal level as
\begin{equation}
  \label{etage}
B_t = \bK [[ Y_1, \ldots , Y_n]] \supset
B_{t+1} = \bK [[ Y_1^p, Y_2, Y_3, \ldots , Y_n]]
\end{equation}
after a suitable choice of regular sequence of local parameters for $B_t$.

Let $I_{t}$ be the maximal ideal of $B_t$.
Now we are ready to prove that the differential that can be described as the natural map
$$
d_\vx \widehat{F} : (I_0/I_0^2)^\ast \longrightarrow (I_\infty/I_\infty^2)^\ast
$$
is surjective. This is equivalent to injectivity of the natural map
$I_\infty/I_\infty^2 \longrightarrow I_0/I_0^2$. Suppose that $d_\vx \widehat{F}$
is not surjective. Then there exists a nonzero $(\alpha_1, \ldots, \alpha_s) \in \bK^s$
such that $Z \coloneqq \sum_j \alpha_j Z_j \in I_0^2$. However, $\widetilde{F}$ is smooth, hence 
$d_{\vx_k} \widetilde{F}$ is surjective and $Z\not\in I_k^2$. Going up the tower,
we can find $t$ such that $Z\not\in I_{t+1}^2$ and $Z\in I_t^2$.
Looking at the description of the floor of the tower in Equation~(\ref{etage}),
we can conclude that $Z \in B_t Y_1^p$. This is a contradiction because
$Z$ is a non-zero polynomial in $X_j$ of degree less than $p$ in each variable.
\end{proof}

It would be quite useful to establish generic smoothness for a larger
class of maps than we currently do in Theorem~\ref{generic_smooth}.
To do that,
more detailed information about the local behaviour of
inseparable maps is essential. By a $p^\bullet$-basis of a ring $R$ over
a subring
$S$ we understand a sequence of elements $a_1, \ldots ,a_n\in R$ together with a sequence
of natural numbers $k_1, \ldots, k_n$ such that
the elements $a_1^{m_1}a_2^{m_2}\ldots a_n^{m_n}$
(where $0 \leq m_i < p^{k_i}$ for all $i$)
form an $S$-basis of $R$. 
\begin{HKC}
  Let $\bK$ be a perfect field of characteristic $p$.
  Consider a higher Frobenius sandwich of commutative local regular $\bK$-algebras
  $$
R \geq  S \geq R^q
$$
where $q=p^s$ for some natural $s$.
Then there should exist
a $p^\bullet$-basis of $R$ over $S$.
\end{HKC}
Certainly one can inquire whether this statement holds for a larger class of rings
$R$ and $S$ but this
is the generality we need. For $s=1$ and regular local rings
this is proved by Kimura and Niitsuma \cite{KiNi}.

We believe that the Higher Kunz Conjecture
is key to the Higher Frobenius Conjecture.

\end{document}